\documentclass{article}
\usepackage[a4paper]{geometry} 
\usepackage{latexsym}         
\usepackage{amsmath,amsthm}          
\usepackage{amssymb}
\usepackage{color}

\newtheorem{lemma}{Lemma}[section]
\newtheorem{proposition}[lemma]{Proposition}
\newtheorem{theorem}[lemma]{Theorem}

\newtheorem{definition}[lemma]{Definition}

\newcommand{\half}{{\textstyle{1\over2}}}
\newcommand{\R}{\mathbb R}%
\newcommand{\Z}{\mathbb Z}%
\newcommand*{\C}{\mathbb C}%
\newcommand{\g}{\mathfrak{g}}
\newcommand{\Sp}{Sp}

\makeatletter
\def\operatorname#1{\mathop{\operator@font #1}\nolimits}%
\makeatother
\newcommand{\Rad}{\mathrm{Rad}\,}%

\DeclareMathOperator{\Tr}{Tr}

\DeclareMathOperator{\id}{Id}

\DeclareMathOperator{\Ker}{Ker}
\DeclareMathOperator{\Image}{Im}
\newcommand{\Span}{\operatorname{Span}}
\DeclareMathOperator{\Id}{Id}
\newcommand*{\cyclic}{\mathop{\kern0.9ex{{+}\kern-2.2ex\raise-.29ex%
      \hbox{\Large\hbox{$\circlearrowright$}}}}\limits}

\DeclareMathAccent{\ring}{\mathalpha}{operators}{"17}

\begin{document}

\markboth{Michel Cahen, Thibaut Grouy and Simone Gutt}
{A possible symplectic framework for Radon-type transforms}

\title{\bf{A possible symplectic framework for Radon-type transforms}}

\author{Michel Cahen$^1$,  Thibaut Grouy and Simone Gutt$^1$\\
 \small{D\'epartment de Math\'ematiques, Universit\'e Libre de Bruxelles}\\ \small{boulevard du triomphe, Campus Plaine CP 218}\\
\small{B-1050 Brussels, Belgium}\\
$^1$\small{Acad\'emie Royale de Belgique}\\
{\scriptsize{mcahen@ulb.ac.be, thibautgrouy@gmail.com, sgutt@ulb.ac.be} }}

\date{}

\maketitle

\begin{abstract}
Our project  is to define Radon-type transforms in symplectic geometry.
The chosen framework consists of symplectic symmetric spaces whose canonical connection is of Ricci-type.
 They can be considered as symplectic analogues of the spaces of constant holomorphic curvature in K\"ahlerian Geometry. They are characterized amongst a class of symplectic manifolds by the existence of many totally geodesic symplectic submanifolds.
We present  a particular class of Radon type tranforms, associating to a smooth compactly supported function on a homogeneous manifold $M$,  a function on a homogeneous  space $N$ of  totally geodesic submanifolds of $M$, and vice versa. We describe some  spaces $M$ and $N$ in such Radon-type duality with $M$  a  model of symplectic symmetric space with  Ricci-type canonical connection and $N$ an orbit of totally geodesic symplectic {submanifolds}.
    \end{abstract}


\section*{Introduction}
The subject of Radon transforms started with results by Funk,
in 1913,  who observed  that a symmetric function on the sphere $S^2$ can be described from its great circle integrals, and by Radon, in 1917, who  showed that a smooth function $f$ on the Euclidean space $\R^3$ can be  determined by its integrals over the planes in $\R^3$: 
if $J(\omega,p)$ is  the integral of $f$ over the plane defined by $x\cdot \omega=p$ for $\omega$  a unit vector, then
$$ f(x)= -\frac{1}{8\pi^2} L_x\Bigl( \int_{S^2} J(\omega,\omega\cdot x )d\omega\Bigr)$$
 where $L$ is the Laplacian.  
This underlines a correspondence between the space $M=\R^3$ and the space of $2$-planes in $\R^3$, $N=Pl(2,\R^3)=\raisebox{0.3mm}{$(S^2\times \R)$}/\raisebox{-0.3mm}{$\Z_2$}$:\\
- the Radon transform associates to  a smooth compactly supported function $f$  on $\R^3$, the function $ \Rad{f}$ on $Pl(2,\R^3)$ defined by
$$
\Rad{f} (S)=\int_{x\in S} f(x) dm(x)\quad  \textrm{ for  any $2$-plane } S;
 $$
 where $dm$ is the Lebesgue measure on $S$.\\
- the dual Radon transform associates to a smooth compactly supported   function $\varphi $  on $Pl(2,\R^3)$, the function $\Rad^*{\varphi} $ on $\R^3$ defined by :
$$
\Rad^*{\varphi} (x)=\int_{S \ni x} \varphi(S) d\mu(S)
$$
where $d\mu$ is the unique normalized measure on $Pl(2,\R^3)$ which is invariant under rotations.\\
 Observing that one can recover a compactly supported function on $\R^3$ from its integral over lines in $\R^3$ led to  applications to X-ray technology and tomography and boosted the interest in this theory. The idea of the Radon transform was widely generalized, constructing correspondances between a class of objects on a space $M$ and a class of objects on a space $N$.

We  consider only Radon-type  transforms between spaces of functions built by integrating functions on totally geodesic submanifolds in a homogeneous framework. This was introduced by Helgason, mostly  in a (pseudo) Riemannian framework, in the  sixties. 
Let us recall that  a connected submanifold $S$ of a  manifold $M$ endowed with a connection $\nabla$ is said to be {\emph {totally geodesic}} if each geodesic in $M$ which is tangent to $S$ at a point lies entirely in $S$.
These submanifolds have been described  for  pseudo-Riemannian space forms $(M,g)$ (endowed with  the Levi Civita connection).
The abstract Radon transform generalizes the original transform,
associating to a smooth compactly supported function $f$ on a homogeneous space $M$ the function $ \Rad f$ {on a homogeneous space $N$ of totally geodesic submanifolds of fixed dimension $p$, whose value at a point $ S \in N$} is given by the integral of $f$ on the corresponding submanifold $S$ 
 \begin{equation} f\rightarrow \Rad{f} (S)=\int_{x \in S} f(x) dm(x).  \end{equation}
The dual abstract Radon transform is :
 \begin{equation}\varphi\rightarrow \Rad^*{\varphi} (x)=\int_{S\supset \{ x\} }\varphi(S) d\mu(S). \end{equation}
The first questions  encountered in constructing such Radon-type transforms deal with finding invariant  {measures} to integrate the functions, relating functions spaces, inverting the transforms, finding a corresponding map
between invariant differential operators, studying the support of $f$ when $\Rad{f}$ has compact support... They have been studied  in the framework of Riemannian spaces of constant curvature,  which provide a rich supply of totally geodesic submanifolds.  The results and the original references on the subject are presented  in  Helgason's nice books \cite{sgbib:Helg1, sgbib:Helgason}. \\

To build such correspondences in a symplectic framework, we choose symplectic symmetric spaces with Ricci-type  canonical connection. We recall their definition in section {\ref{symplectic space forms}}, along with their properties and the construction of models. In section \ref{SymplRadonduality}, we describe  the space of totally geodesic symplectic submanifolds of our models. The corresponding list of homogeneous spaces in symplectic Radon duality was announced  in \cite{sgbib:gutt} when the  Ricci endomorphism of the tangent bundle  squares to a non zero constant multiple of the identity. We extend this list to the case where it squares to zero. \\
In section {\ref{characterization},  we characterize the symplectic locally symmetric spaces with Ricci-type  canonical connection of dimension $\ge8$ as the symplectic manifolds with a symplectic connection $(M,\omega,\nabla)$  having the $2$ following properties :
\begin{enumerate}
\item they admit a  parallel field $A$ of   endomorphisms of the tangent bundle such that  $A_x\in {\mathfrak{sp}}(T_xM,\omega_x)$ and $A^2=\lambda\Id$;
\item  given  any point $x$ and any  $A_x$-stable symplectic subspace $V_x$ of $T_xM$,
 there is a unique maximal totally geodesic symplectic submanifold  passing through $x$ and tangent to $V_x$.
\end{enumerate}
\section{Ricci-type symplectic symmetric spaces   }\label{symplectic space forms}
\subsection{Symplectic connections}
{{A linear connection $\nabla$ on a symplectic manifold $(M,\omega)$  is called
{\emph{symplectic}}}} if the symplectic form $\omega$ is parallel  and 
if its  torsion $T^\nabla$ vanishes.\\
Let us recall that  a symplectic connection exists on any symplectic manifold. Indeed,
there is a canonical projection of any  torsion free connection $\nabla^0$ on a symplectic connection $\nabla$ given as follows : define $N$ by 
$\nabla^0_X\omega (Y,Z)=:\omega(N(X,Y),Z) $ and set
$$\nabla_X Y:=\nabla^0_X Y + \frac{1}{3}N(X,Y)+\frac{1}{3}N(Y,X).$$ 
On the other hand a symplectic connection is far from being unique:
given  $\nabla$ symplectic, the connection 
$\nabla'_X Y:=\nabla_X Y + S(X,Y)$ is symplectic 
if and only if   $\omega(S(X,Y),Z)$ is totally symmetric.

When the manifold is symmetric there is a natural unique connection: it is   the canonical connection for which each symmetry is an affine transformation. In the symplectic context we have precisely:\\
A {\emph{symplectic symmetric space}} is a symplectic manifold $(M,\omega)$ endowed with ``symmetries" i.e. with a smooth map
$$S:M\times M\rightarrow M : (x,y)\mapsto s_xy$$ 
so that, for each $x\in M$, the symmetry $s_x : M\rightarrow M$ is an involutive symplectomorphism  (i.e. $s_x^*\omega=\omega$ and $s_x^2=\Id$), 
with  $x$ an  isolated fixed point (i.e. $s_xx=x$ and $1$ is not an eigenvalue of $(s_x)_{*_x}$), and such that 
$s_xs_ys_x=s_{s_xy}$, for any $x,y\in M$.\\
On  a symplectic symmetric space there is a unique  connection for which each $s_x$ is an affine transformation; it is given by 
\begin{equation}
(\nabla_XY)_x=\half [X-s_{x*}X,Y]_x.
\end{equation}
This connection is automatically symplectic. Symmetric symplectic spaces were introduced in \cite{sgbib:BCG, sgbib:Biel}.

The curvature tensor of a symplectic connection at a point $x$ has the following symmetry properties:
\[
\omega_x (R_x^\nabla(X,Y)Z,T)=-\omega_x (R_x^\nabla(Y,X)Z,T)=\omega_x (R_x^\nabla(X,Y)T,Z) 
\]
and 
\[
\cyclic_{XYZ} R_x^\nabla(X,Y)Z=0 \quad , \cyclic_{XYZ}{\footnotesize { \textrm{ denoting the sum over cyclic permutations of  } {X,Y}\textrm{ and }{Z}}}.
\]
When $dim M=2n\ge 4$, Izu Vaisman \cite{sgbib:Vaisman} showed that the space of tensors having those symmetries splits under the action of the symplectic group into two irreducible components
so that one has a decomposition of the curvature into
\begin{equation}\label{eq:decR}
R^\nabla = W^\nabla + E^\nabla 
\end{equation}
where $W^\nabla$ has no trace and 
\begin{eqnarray}
E^\nabla(X,Y)Z &=& 
{\textstyle{\frac1{2n+2}}}\left(2\omega(X,Y)\rho^\nabla Z
 + \omega(X,Z)\rho^\nabla Y - \omega(Y,Z)\rho^\nabla X\right.\cr
&&\quad \left. + \omega(X,\rho^\nabla Z)Y - \omega(Y,\rho^\nabla Z)X
\right)
\end{eqnarray}
 with  the Ricci tensor $r^\nabla$ (defined by $r^\nabla(X,Y):=\Tr[Z\rightarrow R^\nabla(X,Z)Y]$), which is automatically symmetric,  converted into the Ricci endomorphism $\rho^\nabla$  by
\begin{equation}
\omega(X, \rho^\nabla Y) = r^\nabla(X,Y).
\end{equation}
A symplectic connection for which $W^\nabla=0$ is said to be of {\emph {Ricci-type}}.

 \subsection{Space forms}
Let us recall that a {\emph{Pseudo Riemannian space form} } is a   connected  pseudo Riemannian manifold $( N,\ring g)$ of dimension $n\ge 4$, which is geodesically complete (for the Levi Civita connection), and   which has  {constant sectional curvature} $k$.\\
Its curvature
 $\ring R$ is then given by 
\begin{equation}
\ring g_x(\ring R_x(X,Y)Z,T)= k\left( \ring g_x(X,Z)\ring g_x(Y,T)-\ring g_x(X,T)\ring g_x(Y,Z)\right);
\end{equation}
the space $( N,\ring g)$ is thus locally symmetric ($\ring\nabla \ring R=0$) and the  curvature is  a polynomial in the tensor algebra only
 involving the metric tensor $\ring g$.\\
Riemannian space forms in dimension $n$ are quotients of the Euclidean space $\R^n$, the sphere $S^n$ or the hyperbolic space $H^n$.\\

If a symplectic manifold $(M,\omega)$ is endowed with a symplectic connection whose  curvature is  a polynomial in the tensor algebra only
 involving the symplectic  tensor $\omega$, then the curvature is invariant under the symplectic group and, in view of  the decomposition formula \eqref{eq:decR}, it is identically zero. In the definition of  symplectic space forms, we obviously want to go beyond the flat case, so we consider the K\"ahler case to generalize it.\\

A {\emph{K\"ahler Space form}} is a connected  K\"ahler   manifold $( N,\ring g, \ring J)$ of dimension $n\ge 4$, which is geodesically complete (for the Levi Civita connection) and   which has {constant holomorphic sectional curvature} $k$ (i.e.:
$
\ring g_x(\ring R_x(X,JX)X,JX)=k\left(\ring  g_x(X,X)^2\right) 
$ 
 for all $x\in N$ and for all $ X\in T_x N $). \\
Its curvature is then given by 
\begin{eqnarray}\label{eq:curvkahler}
{\small{\ring g_x(\ring R_x(X,Y)Z,T)}}&=&\small{\frac{k}{4}\Bigl(\ring  g_x(X,Z)\ring g_x(Y,T)-\ring g_x(X,T)\ring g_x(Y,Z)+\ring g_x(X,JZ)\ring g_x(Y,JT)\Bigr.}\nonumber\\
&&
\small{\Bigl.-\ring g_x(X,JT)\ring g_x(Y,JZ) +2\ring g_x(X,JY)\ring g_x(Z,JT)\Bigr)};
\end{eqnarray}
 the space is thus locally symmetric and the  curvature is  a polynomial in the tensor algebra 
 (with catenation) involving only   $\ring g$   and  $J $.\\
 K\"ahler space forms in complex dimension $n$ are all  quotients of $\C^n$, the complex projective space $ \C \mathbb{P}^n$, or the complex hyperbolic space $\C \mathbb{H}^n$.
 
 Remark that the K\"ahler form $\omega_x(X,Y):=\ring g_x(X,JY)$ is symplectic and the Levi Civita connection is symplectic. In symplectic terms,
 formula \eqref{eq:curvkahler} rewrites as
{\small{$
\ring R_x(X,Y)Z=\frac{k}{4}\bigl(-\omega_x(X,JZ)Y+\omega_x(Y,JZ)X-\omega_x(X,Z)JY+\omega_x(Y,Z)JX -2\omega_x(X,Y)JZ\bigr).
$}}
Hence the Levi Civita connection is of Ricci type with Ricci endomorphism equal to a multiple of the complex structure : $\ring\rho=-\frac{k(n+1)}{2}J$.
 
 \begin{definition}
A {\bf{ symplectic space form}}  is a  connected symplectic manifold  endowed with
a symplectic  connection $\nabla$ which is complete, locally symmetric, and such 
 that its curvature is of Ricci type.
 \end{definition}
 The curvature in a symplectic space form is thus given by 
{\begin{eqnarray}\label{eq:RRiccitype}
R^\nabla(X,Y)Z &=& 
{\textstyle{\frac1{2n+2}}}\biggl(2\omega(X,Y)\rho^\nabla Z
 + \omega(X,Z)\rho^\nabla Y - \omega(Y,Z)\rho^\nabla X\cr
&&+ \omega(X,\rho^\nabla Z)Y - \omega(Y,\rho^\nabla Z)X
\biggr)
\end{eqnarray}}
with $\rho^\nabla$ the Ricci endomorphism. The condition to be locally symmetric is equivalent to $\nabla\rho^\nabla=0$, thus
$\rho^\nabla$ commutes with $R^\nabla(X,Y)$ and this yields
$\omega(X,Z)(\rho^\nabla)^2 Y - \omega(Y,Z)(\rho^\nabla)^2 X= \omega(X,(\rho^\nabla)^2 Z)Y - \omega(Y,(\rho^\nabla)^2 Z)X$
which implies
$(\rho^\nabla)^2=k\Id$ where $k$ is a real constant. 

Symmetric symplectic  spaces with Ricci-type connections were studied with John Rawnsley in \cite{sgbib:CGR};
Nicolas Richard examined the analogue of the notion of constant holomorphic sectional curvature in a symplectic context  in his thesis \cite{sgbib:Richard}.
\subsection{Models of Ricci-type symmetric symplectic spaces}
We recall a  construction by reduction of examples of Ricci-type symplectic connections  which was given by Baguis and Cahen in  \cite{sgbib:BC}.
Let $(\R^{2n+2}\setminus\{ 0\},\Omega)$ be the standard symplectic vector space without the origin
and let $A$ be an element in the symplectic Lie algebra $sp(\R^{2n+2},\Omega)$, i.e.$^{tr}A\Omega+\Omega A=0$, where $^{tr}B$ denotes the transpose of the matrix $B$. We consider the reduction with respect to the action of the $1$-parameter subgroup $\exp tA$; the action is Hamiltonian with corresponding Hamiltonian $h(x)=\half \Omega(x,Ax)$ and we
consider the embedded hypersurface given by a level set:
$\Sigma_A=\{ x\in \R^{2n+2} \,\vert\, \Omega(x,Ax)=1\}$,  assuming it to be non empty.
The $1$-parameter subgroup $\{\,{\exp tA}\,\}$ acts on $\Sigma_A$ and we consider the quotient
 $$ M^{red}:=\raisebox{0.3mm}{$\Sigma_A $}/\raisebox{-0.3mm}{${\exp tA}$} \quad \textrm{( it exists at least locally since $Ax\neq 0$!)}$$
with the canonical projection
$$ \pi:\Sigma_A\rightarrow M^{red}.$$
For any $x\in\Sigma_A$, we define a $2n$-dimensional  ``horizontal" subspace $H_x$ of the tangent space $T_x\Sigma_A$
given by the $\Omega$-orthogonal to the subspace spanned by $x$ and $Ax$:
 $$
 H_x:=\Span\{x,Ax\}^{\perp_\Omega}\,\subset \, T_x\Sigma_A\simeq \,\Span\{Ax\}^{\perp_\Omega}\, \subset\,  \R^{2n+2};$$
 the differential of the projection $\pi$ induces an isomorphism 
 $$\pi_{*_x}:H_x \sim T_{\pi(x)}M^{red}.
 $$ 
 Given a tangent vector $X\in T_yM^{red}$, we denote by $\overline{X}$ its horizontal lift:
$${\overline{X}}_x\in H_x, ~~ X_{y=\pi(x)}=\pi_{*_x}{\overline{X}}_x.$$
The reduced $2$-form $\omega^{red}$ on $M^{red}$ is defined in the standard way as the unique $2$-form $\omega^{red}$ on $M^{red}$
such that :
\begin{equation}
 \pi^*\omega^{red}=\Omega_{\vert_{\Sigma_A}}\quad\quad \textrm {i.e. }\quad  \omega^{red}_{y=\pi(x)}(X,Y):=\Omega_x({\overline{X}}_x,{\overline{Y}}_x);
 \end{equation}
and $(M^{red},\omega^{red})$ is a symplectic manifold.
The flat connection $\ring\nabla$ on $\R^{2n+2}$ induces a connection 
$\nabla^{red}$ on $M^{red}$ given by
\begin{equation}\label{eq:nablared}
{{(\nabla^{red}_X Y)_y:=\pi_{*_x}(\ring\nabla_{\overline{X}}{\overline{Y}}
-\Omega(A{\overline{X}},{\overline{Y}})x +\Omega ({\overline{X}},{\overline{Y}})Ax)}}.
\end{equation} 
This reduced connection $\nabla^{red}$ on $( M^{red},\omega^{red})$  is  symplectic and of Ricci-type.
With Lorenz Schwachh\"ofer, we proved in \cite{sgbib:CGS} that any symplectic manifold endowed with a symplectic
connection of Ricci-type is locally of this form.
In the above construction the Ricci endomorphism at $y=\pi(x)\in M^{red}$ is proportional to the map induced by $A$ on $H_x$ :
$$-\frac{1}{2n+2} \rho^{\nabla^{red}}_{\pi(x)}({{X}})=\pi_{*_x}\left(A{\overline{X}}-\Omega(A{\overline{X}},Ax)x\right).
$$
The space is   locally symmetric if and only if  $(\rho^\nabla)^2=k\Id$ and this happens if and only if $A^2=\lambda\Id$.\\
The  models $(M_A,\omega^{red},\nabla^{red})$ of symplectic space forms that we shall use are the connected components of the ones obtained by the above construction,  using an element $A \in sp(\R^{2n+2},\Omega)$ so that $A^2=\lambda\Id$; in these cases the quotient of the hypersurface $\Sigma_A$ by the action of $\{ \exp tA\,\}$ is globally defined
and we have
$$\pi :  \Sigma_A=\{ x\in \R^{2n+2}\,\vert\, \Omega(x,Ax)=1\} \rightarrow M_A:=M^{red}_{cc}=(\raisebox{0.3mm}{$\Sigma_A$} /\raisebox{-0.3mm}{${\{\exp tA\}}$})_{cc}$$
where $~_{cc}$ indicates that we take a connected component.\\
Any $B\in \widetilde{G}_A:=\{ \, B\in Sp(\R^{2n+2},\Omega)\, \vert \, BA=AB \, \}_0$, where $_0$ denotes the connected component of the identity, induces an automorphism of $M_A$ in the obvious way:
$$
B\cdot \pi(x):=\pi(Bx).
$$
 The space $M_A$ is not only locally symmetric; it is a symplectic symmetric space, with the symmetry at $y=\pi(x)$  induced by
$$
S_x(v)=-v-2\Omega(v,x) Ax +2\Omega(v,Ax) x.
$$
\begin{theorem}\cite{sgbib:CGS,sgbib:CGW}
Let $A \in sp(\R^{2n+2},\Omega)$ be so that $A^2=\lambda\Id$ and consider the reduced manifold $M_A=(\{ x\in \R^{2n+2}\vert \Omega(x,Ax)=1\} /_{\{\exp tA\}})_{cc}$ with the reduced symplectic structure $\omega^{red}$ and the reduced symplectic connection $\nabla^{red}$.\\
- Each $(M_A,\omega^{red},\nabla^{red})$ is a symmetric symplectic Ricci-type manifold;
\begin{enumerate}
\item it admits the group $\widetilde{G}_A=\{ \, B\in Sp(\R^{2n+2},\Omega)\, \vert \, BA=AB \, \}_0$, as well as the quotient group $\left(\raisebox{0.3mm}{$\widetilde{G}_A$}/\raisebox{-0.3mm}{$\{\exp tA\}$}\right)_0$, as  groups of automorphisms acting transitively;
\item the action of $\widetilde{G}_A$ is strongly Hamiltonian : if $D$ is an element in the Lie algebra ${\widetilde{\g}}_A$ of $\widetilde{G}_A$ and if $D^{*M_A}$ denotes the corresponding fundamental vector field on $M_A$, i.e. $D^{*M_A}_y=\frac{d}{dt}_{\vert_0} \exp(-tD)\cdot y$ then $\iota(D^{*M_A})\omega^{red}=df_D$ with $\left(\pi^*(f_D)\right)(x)=\half(\Omega(x,Dx)$ and the associated moment map $$J : M_A\rightarrow {\widetilde{\g}}_A^* : \pi(x)\mapsto [D\to \half(\Omega(x,Dx)]$$ is $\widetilde{G}_A$ equivariant.
\item Any symplectic space form is diffeomorphic to a quotient of the universal cover of  a model space $(M_A,\omega^{red},\nabla^{red})$. 
\end{enumerate}
\end{theorem}
A description of the model spaces appeared in \cite{sgbib:CGS}. We now refine this description as follows : 
\begin{proposition}
{\bf{Case 1: $A^2=\lambda\Id$ with $\lambda>0$}}\\
If $\lambda=k^2, k>0$ we view $\R^{2n+2}=L_+\oplus L_-$ as the sum of the two Lagrangian subspaces  corresponding to the $\pm k$ eigenspaces for $A$. In an adapted basis where $A=\left(\begin{matrix}k\Id&0\\0&-k\Id\end{matrix}\right)$  and $\Omega=\left(\begin{matrix}0&\Id\\ -\Id&0\end{matrix}\right)$, the level set is given by $$\Sigma_A=\{ (u,v) \in \R^{2(n+1)}\,\vert\, u\cdot v =-\frac{1}{2k}\}.$$ Seeing the cotangent bundle to the sphere as 
$$T^*S^n=\{( \tilde{u},\tilde{v})\in \R^{2(n+1)}\,\vert\, \Vert \tilde{u} \Vert=1, \tilde{u}\cdot \tilde{v}=0\},$$ the canonical cotangent symplectic structure is given by the restriction to $T^*S^n$ of the $2$-form $\tilde{\Omega}$ on $ \R^{2(n+1)}$ defined by $\sum_{i=1}^{n+1} d\tilde{v}^i\wedge d\tilde{u}^i$. \\The map
 $\phi: \Sigma_A\rightarrow T^*S^n: (u,v) \mapsto (\tilde{u}=\frac{v}{\Vert v\Vert},\tilde{v}=\Vert v \Vert u +\frac{1}{2k} \frac{v}{\Vert v\Vert})$, induces  a diffeomorphism    $$\psi : M^{red}\rightarrow T^*S^n : \pi(u,v)=\pi(e^tu,e^{-t}v)\mapsto (\tilde{u}=\frac{v}{\Vert v\Vert},\tilde{v}=\Vert v \Vert u + \frac{1}{2k} \frac{v}{\Vert v\Vert}).$$
The pullback by this diffeomorphism of the  canonical symplectic structure on $T^*S^n$ is  the reduced symplectic structure $\omega^{red}$ on $M^{red}$; indeed 
\small{
$$\pi^*(\psi^*(\tilde{\Omega}_{\vert_{T^*S^n}}))=\phi^*(\tilde{\Omega}_{\vert_{T^*S^n}})=\left(\sum_{i=1}^{n+1} du^i\wedge d{v}^i -\frac{d(u.v)\wedge (v.dv)}{\Vert v\Vert^2}\right)_{\vert_{\Sigma_A}}=\Omega_{\vert_{\Sigma_A}}=\pi^*\omega^{red},
$$}
 and $\pi^*$ is injective. Since $T^*S^n$ is connected, $M_A=M^{red}$, and, as a symplectic manifold, we have
\begin{equation}
 M_A\simeq T^*S^n. 
 \end{equation}
 Observe that it is simply connected when its dimension is at least $4$.\\A matrix $B$ is in $\widetilde{G}_A$ if and only if it is of the form $B=\left(\begin{matrix}C&0\\0&^{tr}C^{-1}\end{matrix}\right)$ with $C\in Gl_+(n+1,\R)$ so that  $\widetilde{G}_A\simeq Gl_+(n+1,\R)$.
 The group $G_A=\raisebox{0.3mm}{$\widetilde{G}_A$}/\raisebox{-0.3mm}{$\{\exp tA\}$}$ identifies with the subgroup  $Sl(n+1,\R)$; it acts transitively on $M_A$, and the  symmetric structure is defined by
\begin{equation}
M_A= \raisebox{0.3mm}{$Sl(n+1,\R)$}/\raisebox{-0.3mm}{$Gl_+(n,\R)$}
\end{equation} 
with $Gl_+(n,\R)$ sitting in $Sl(n+1,\R)$ as $\left\{\left(\begin{matrix} \det B^{-1}&0\\0&B\end{matrix}\right)\,\vert\, B\in Gl_+(n,\R)\right\}$; it is the stabilizer of the point $p_0:=\pi(-\frac{1}{2k}e_1,e_1)$ for $e_1=(1,0,\ldots,0) \in \R^{n+1}$; the symmetry $s_{p_0}$ is induced by $S_0:=\left(\begin{matrix} 1&0\\0&-\Id\end{matrix}\right)$ and conjugation by $S_0$ is an automorphism $\sigma$ of $Sl(n+1,\R)$ with $Gl_+(n,\R)$ the connected component of its fixed points.\\
The action of  $Sl(n+1,\R)$ is strongly Hamiltonian; identifying ${\mathfrak{sl}}(n+1,\R)^*$  to ${\mathfrak{sl}}(n+1,\R)$ via the Killing form, the associated moment map is $$ J': M_A\mapsto {\mathfrak{sl}}(n+1,\R) : \pi(u,v)\mapsto \frac{1}{2(n+1)}\left(- u\otimes (v\cdot \, )+\frac{1}{n+1}(u\cdot v)\Id\right);$$ it 
presents the homogeneous symplectic manifold $M_A$ as a double cover of the adjoint orbit in  ${\mathfrak{sl}}(n+1,\R)$ of the element $\frac{1}{4k(n+1)^2}\left(\begin{matrix} n&0\\0&-\Id\end{matrix}\right)$; this orbit  is $\raisebox{0.3mm}{$Sl(n+1,\R)$}/\raisebox{-0.3mm}{$Gl(n,\R)$}$.\\

{\bf{Case 2: $A^2=\lambda\Id$ with $\lambda<0$}}\\ 
If $\lambda=-k^2<0$, then    $J=\frac{1}{k} A$ defines a complex structure, identifying $\R^{2n+2}$ to $\C^{n+1}$. We denote by $p$ the integer such that the signature of the quadratic form $\Omega(x,Ax)$ on $\R^{2(n+1)}$ is $(2(p+1), 2(n-p))$; the pseudo-Hermitian
structure  defined by $<x,y>:=\Omega(x,Jy)-i\Omega(x,y)$ on $\C^{n+1}$ is of signature (over $\C$) equal to $(p+1,n-p)$. Clearly $0\le p\le n$.  The level set is given by $\Sigma_A=\{ z\in \C^{n+1}\,\vert\, <z,z>=\frac{1}{k}\}$ and the reduced space $M^{red}=\{ [z]\,\vert\, z \in \C^{n+1}, <z,z>=\frac{1}{k}, [z]=[e^{it}z] \}$ is connected and simply connected.  A $\R$-linear endomorphism of $\R^{2(n+1)}$ which commutes with $J$ identifies  with a $\C$- linear endomorphism of $\C^{n+1}$. The group $\widetilde{G}_A$ identifies with the pseudo-unitary group $U(p+1,n-p)= Sp(\R^{2(n+1)},\Omega)\cap Gl(n+1,\C)$ and the quotient $\raisebox{0.3mm}{$\widetilde{G}_A$}/\raisebox{-0.3mm}{$\{\exp tA\}$}$ to a finite quotient (by $(n+1)$th roots of the identity) of  its subgroup $SU(p+1,n-p)$ and we denote in this case by $G_A=SU(p+1,n-p)$. The model space is 
\begin{equation}
M_A=M^{red}=\raisebox{0.3mm}{$SU(p+1,n-p)$}/ \raisebox{-0.3mm}{$U(p,n-p)$},
\end{equation}
endowed with its natural invariant pseudo K\"ahler structure, 
with $U(p,n-p)$ sitting in $SU(p+1,n-p)$ as $\{\left(\begin{matrix} \det B^{-1}&0\\0&B\end{matrix}\right)\,\vert\, B\in U(p,n-p))\}$ i.e. as the stabilizer of the point $p_0:=\pi(e_1)$ for $e_1=(1,0,\ldots,0) \in \C^{n+1}$.
Thus 
\begin{itemize}
\item $M_A=\C H^n=\raisebox{0.3mm}{$SU(1,n)$}/\raisebox{-0.3mm}{$U(n)$}$ is the complex hyperbolic space for $p=0$;
\item $M_A=E$ is  a holomorphic vector bundle of rank $n-p$ over $\C P^{p}$ for $1\le p\le n-1$;  
\item $M_A=\C P^n=\raisebox{0.3mm}{$SU(n+1)$}/ \raisebox{-0.3mm}{$U(n)$}$ is the complex projective space  for $p=n$.
\end{itemize}
The action of  $SU(p+1,n-p)$ is strongly Hamiltonian; using  the trace to identify ${\mathfrak{su}}(p+1,n-p)^*$  with ${\mathfrak{su}}(p+1,n-p)$, the moment map is
$$ J': M_A\mapsto {\mathfrak{su}}(p+1,n-p) : \pi(z)\mapsto \left(\frac{-i}{2} z\otimes <\cdot ,z\,> +\frac{i<z,z>}{2(n+1)}\Id\right);$$ it 
presents the homogeneous symplectic manifold $M_A$ as the adjoint orbit of $\frac{1}{2k(n+1)}\left(\begin{matrix} -ni&0\\0&i\Id\end{matrix}\right)$ in ${\mathfrak{su}}(p+1,n-p)$.

{\bf{Case 3: $A^2=0$}}\\ If $\lambda=0$ there are two integers $r$ and $p$ attached to the space form, $r$ being the rank of $A$ (so that $1\le r\le n+1$) and $( p,r-p )$ being the signature  of the quadratic $2$-form $\Omega(x,Ax)$ (so that $1\leq p\leq r$) which induces a pseudo Riemannian metric $g$ on the quotient 
$$V:=\raisebox{0.3mm}{$\R^{2(n+1)}$}/\raisebox{-0.3mm}{$\Ker A$}.$$
 Observe that $\Ker A$ is the $\Omega$-orthogonal of $\Image A$, and  $\Image A \subset \Ker A$ so that 
 $$
 W:=\raisebox{0.3mm}{$\Ker A$}/\raisebox{-0.3mm}{$\Image A$}
 $$
 with the  $2$-form $\Omega'$ induced by $\Omega$ is a $2m :=2(n+1-r)$-dimensional symplectic vector space. Choosing a $r$-dimensional $\Omega$-isotropic subspace $V'$ supplementary to $\Ker A$, we have $V\simeq V'$ and we denote by $g'$ the corresponding metric of signature $(p,r-p)$ on $V'$. If $W'$ is the $\Omega$-orthogonal to $AV'\oplus V'$, then $W'\simeq W$, and we write 
 $\R^{2(n+1)}= (AV'\oplus V') \oplus^{\perp_\Omega} W'$ ; in well chosen corresponding basis, \\
 $A=\left(\begin{matrix} 0&\Id_r&0\\ 0&0&0\\0&0&0\end{matrix}\right)$ and $\Omega=\left(\begin{matrix} 0&-G&0\\ 
G&0&0\\0&0&\Omega'\end{matrix}\right)$, with $G:=\Id_{p}\oplus -\Id_{r-p}$, so the level set is \\$\Sigma_A=\{(Ax,v,w)\,\vert\, x,v \in V', w\in W', g'(v,v)=\Omega(v,Av)=1\}$. 
We define $$Q^{p,r-p}:=\{ v\in V\simeq \R^{r}\,\vert \, g(v,v)=1\}$$ and identify its tangent bundle to
$TQ^{p,r-p}=\{ v,y\in V\,\vert \, g(v,v)=1,g(y,v)=0\}$. The map 
 $$\phi : M^{red}\rightarrow T^*Q^{p,r-p}\times W : \pi(Ax,v,w)\mapsto  \left(v, -g(x,\cdot),w\right)$$
 is well defined since $\exp tA\cdot (Ax,v,w)=(Ax+tAv,v,w)$ and defines a symplectomorphism
 \begin{equation}
 M^{red}\simeq T^*Q^{p,r-p}\times W
 \end{equation}
with the product  of the canonical cotangent bundle symplectic structure on $T^*Q^{p,r-p}$ and the natural symplectic structure on $W$.
Remark that $M^{red}$ is connected iff $p>1$, and has two connected components if $p=1$; it is simply connected when $p\neq 2$.

\noindent The group of symplectic transformations of $\R^{2(n+1)}$ commuting with $A$, written in the basis described above, is

{\scriptsize{$$
\left\{ \left(\begin{matrix} B&BGT&BU\Omega'\\ 
0&B&0\\
0& -C(^{tr}U)G&C\\
\end{matrix}\right)=\left(\begin{matrix} B&0&0\\ 
0&B&0\\
0& 0&C\\
\end{matrix}\right) \left(\begin{matrix} \id&GT&U\Omega'\\ 
0&\Id&0\\
0& -^{tr}UG&\id\\
\end{matrix}\right)\,{\Big{\vert}}\, \begin{matrix}B \in O(p,r-p), U\in Mat(r\times 2m,\R)\\
                                                            \ T=S+\half GU\Omega' (^{tr}U)G,\\ C\in \Sp(2m,\R), S\in Sym(r,\R)\end{matrix}
\right\}
$$}}
so that $\widetilde{G}_A$ is the semidirect product of the semisimple subgroup $S$, which is the direct product $S:= SO_0(p,r-p)\times Sp(2m, \R)$, and the normal solvable subgroup $R=Mat(r\times 2m,\R)\cdot Sym(r,\R)$ where $Sym(r,\R)$ is the set of symmetric real $r\times r$ matrices, with the product 
$$(U,S)\cdot (U',S')=\left(U+U', S+S'+\half G(U'\Omega'(^{tr}U)-U\Omega'(^{tr}U'))G\right).$$ 
As symmetric space, we have
{\scriptsize{\begin{equation}
M_A= \raisebox{0.5mm}{$S\cdot R$}/ \raisebox{-0.5mm}{$ \bigl(SO_0(p-1,r-p)\times Sp(2m, \R)\bigr)\cdot \bigl(Mat((r-1) \times 2m,\R)\cdot (\R\times Sym((r-1),\R))\bigr)$}
\end{equation}}}
which is the quotient of ${\widetilde{G}}_A$ by the stabilizer of $\pi(0,e_1,0)$, with $e_1$ the first vector in the chosen pseudo-orthogonal basis of $V$, thus the set of elements of $\widetilde{G}_A$ mapping $(0,e_1,0)$ on an element of the form $(te_1,e_1,0)$ for some $t\in\R$, i.e. those elements $(B,C,U,S)$ described above with $Be_1=e_1, ^{tr}Ue_1=0$ and $Se_1=te_1$.\\
Denoting by $(b,c,u,s)$ an element in the Lie algebra $\widetilde{\mathfrak{g}}_A$, the moment map 
$$
J:M_A\simeq\left(T^*Q^{p,r-p}\times W\right)_{cc}\rightarrow \widetilde{\mathfrak{g}}_A^*
$$ reads $<J(\pi(Ax,v,w)),(b,c,u,s)>= -g(x,bv)+ \half s(v,v)+g(v,U\Omega'w)+ \half\Omega(w,cw)$.\\
Hence $J(\pi(Ax,v,w))=J(\pi(Ax',v,'w'))$ iff $(Ax',v',w')=\pm(Ax,v,w)$, so the moment map is a double cover of the  orbit of the element $\half s_{11}^*\in \widetilde{\mathfrak{g}}_A^*$ defined by $s^*(b,c,u,s)=s^{11}$
when $p>1$.
When $p=1, M_A$ is one of the connected component of $T^*Q\times W$ hence it corresponds to elements in $v\in Q$ so that  the first component has a given sign, i.e. $ v^1>0$ for one component and $v^1<0$ for the other component. In that case, if $\pi(Ax,v,w)$is in $M_A$ then
$\pi(-Ax,-v,-w)$ is not in $M_A$ so $J$ is a symplectomorphism between $M_A$ and  the coadjoint orbit of $\half s_{11}^*$ in $\widetilde{\mathfrak{g}}_A^*$.

\end{proposition}
\section{Homogeneous spaces in symplectic Radon duality}\label{SymplRadonduality}
\subsection {Totally geodesic symplectic submanifolds in $M_A$ }
We first show that our model spaces possess  as many as possible symplectic totally geodesic submanifolds. 
\begin{theorem}
  Let $M_A:= \left( \raisebox{0.3mm}{$\Sigma_A$}/ \raisebox{-0.3mm}{$\{\exp tA\}$}\right)_{cc} $ be a model of symplectic space form of dimension $2n$ as constructed above, with its canonical symmetric connection which is of Ricci type. We denote  as before by $\rho$ the corresponding Ricci endomorphism.
\begin{enumerate}
\item \label{thmpoint1} Let $S$ be a totally geodesic symplectic submanifold of $M_A$, of dimension $2q$, passing through $y$ and let $V=T_yS\subset T_yM_A$.
  Then $V$ is a symplectic subspace of $T_y( M_A)$ stable by $\rho_y$.\\
\item \label{thmpoint2} Reciprocally, let $V$ be a $2q$-dimensional symplectic subspace of $T_y( M_A)$ stable by $\rho_y$.
There exists a unique maximal  totally geodesic submanifold  $S$ of $M_A$, of dimension $2q$, passing through $y$ and tangent to $V$. It is given by
 $$
   S=  \left(  \raisebox{0.3mm}{$ \Sigma_A\cap W$} /  \raisebox{-0.3mm}{$\{\exp tA\} $}\right)_{cc} \quad \mathrm{with}\quad W= \overline{V}\oplus \R x\oplus \R Ax
 $$
 where $x$ is a point in $\Sigma_A$ so that  $\pi(x)=y$ and where $\overline{V}$ the $2q$--subspace of $\R^{2n+2}$
which is the horizontal lift of $V$ in  $T_x( \Sigma_A)=\, \Span\{Ax\}^{\perp_\Omega}\subset \R^{2n+2},$ i.e. the subspace defined by $\Omega(\overline{V},x)=0$,
 $\Omega(\overline{V},Ax)=0$ and $\pi_{*x} \overline{V}=V$.\\
Such a  totally geodesic submanifold is  automatically  a symplectic space form.  
\end{enumerate}
\end{theorem}
\begin{proof}
The proof follows from the fact  that a submanifold $S$ in a manifold  endowed with a torsionfree connection $(M,\nabla)$ is totally geodesic if and only if, for any $y\in S$ and any $X\in T_yS$, $\nabla_XY(y)$ belongs to $T_yS$ whenever $Y$ is tangent to $S$.\\
 In particular, if $S$ is a totally geodesic submanifold
$R^\nabla_y(X,Y)Z\in T_yS$ for any $y\in S$ and $X,Y,Z \in T_yS$. In view of the expression of a Ricci-type curvature given by
\eqref{eq:RRiccitype}, the tangent space $T_yS$ at the point $y\in S$ to a symplectic totally geodesic submanifold  $S$ in a Ricci-type symplectic manifold
$(M,\omega,\nabla)$ is stable by $\rho^\nabla_y$ and this proves (\ref{thmpoint1}).\\
To prove (\ref{thmpoint2}), we observe that, for any $\tilde{x}\in W$ and any $W$-valued vector fields ${\widetilde{X}},{\widetilde{Y}}$ on $W$, then 
$$
\ring\nabla_{\widetilde{X}}{\widetilde{Y}}({\tilde{x}})
-\Omega(A{\widetilde{X}},{\widetilde{Y}}){\tilde{x}} +\Omega ({\widetilde{X}},{\widetilde{Y}})A{\tilde{x}} \quad \textrm{ is in } \, W,
$$
{since the stability of $V$ by $\rho_y$ implies the stability of $W$ by $A$.}
Using formula \eqref{eq:nablared} which describes the reduced connection, $\nabla^{red}_XY(y)$ belongs to $T_yS$ for all $y\in S, X\in T_yS$ and $Y$ is tangent to $S$ when $ {{S=  \left(   \raisebox{0.3mm}{$\Sigma_A\cap W$}/\raisebox{-0.3mm}{$\{\exp tA\} $}\right)_{cc}}}$.
\end{proof}
  
\subsection{Spaces in Radon-type duality}
The group $G_A:=\left(\{ \, B\in Sp(\R^{2n+2},\Omega)\, \vert \, BA=AB \, \}/\{\exp tA\}\right)_0$ acts by symplectic affine transformations on $(\,M_A,\,\omega^{red},\,\nabla^{red}\,)$; it maps a symplectic totally geodesic submanifold of dimension $2q$ on a symplectic totally geodesic submanifold of dimension $2q$. 
\begin{theorem}  Let $M_A:= \left(\raisebox{0.3mm}{$\Sigma_A$}/\raisebox{-0.3mm}{$\{\exp tA\}$}\right)_{cc} $ be a model of symplectic space form of dimension $2n$ as described before.
\begin{enumerate}
  \item {{There exists a finite number of orbits of $G_A$ in the set of  symplectic maximal  totally geodesic 
      submanifolds of  $M_A$ for any given dimension $2q$.}}\\ 
     \item {{Each of these $G_A$--orbits 
  is a symmetric space.}}  \\
  \item {If $A^2\neq 0$, those orbits are symplectic symmetric spaces. }
   \end{enumerate}
  \end{theorem}
  \begin{proof}
  Point (1) follows from the fact that the action of $G_A$  on  the space  of symplectic maximal totally geodesic submanifolds of $M_A$ 
corresponds bijectively to  the action of $G_A$ on the set  of $(2q+2)$-dimensional symplectic subspaces $W$ of $\R^{2n+2}$ which are stable by $A$ and intersect $\Sigma_A=\{ x\in \R^{2n+2}\,\vert \,\Omega(x,Ax)=1\}$.\\
To see point (2), one observes that  given a subspace $W$, the conjugation by $\Id_W\oplus -\Id_{W^{\perp_\Omega}}$ is an automorphism of $\widetilde{G}_A$. It induces an automorphism of $G_A$. The fixed points of this automorphism are  the elements which map $W$ in $W$. \\
Point (3) results from the precise description given below.
  \end{proof}
  We consider the  {{ Radon transform}} described in the introduction,  choosing  one orbit $N$ of the automorphism group $G_A$ in the set of symplectic maximal totally geodesic submanifolds of $M_A$. 
One associates to a continuous function $f$ on $M_A$, with  compact support,   the function $\Rad{f}$ on $N$ defined by
  $
  \Rad{f}(S)=\int_{x\in S} f(x) d\mu (x)
  $
  with $d\mu$ an invariant measure on the totally geodesic submanifold $S$ (which exists since $S$ is a symplectic space form).
 The {dual Radon transform }associates to a continuous function $F$ on $N$, with compact support, the function $\Rad^* F$ on $M_A$ ,
   $
  \Rad^*{F}(x)=\int_{S\ni x} F(S) d\nu (S)
  $
  with $d\nu$ an invariant measure on $N$ (which can be shown to exist from the explicit description below).\\
  The spaces $M_A$ and $N$ which are in such Radon-type duality are  the following :
\begin{proposition}
The $G_A$-orbits of maximal totally geodesic symplectic submanifolds in the models spaces $M_A$, when $A^2\neq 0$, are given as follows. 
 \begin{itemize}
 \item When $A^2=k^2\Id$, we view as before $\R^{2n+2}=L_+\oplus L_-$ as a sum of two Lagrangian subspaces  corresponding to the $\pm k$ eigenspaces for $A$.
 The group  $G_A$ is $Sl(n+1,\R)$ and the model space form is  the cotangent bundle to the sphere with its canonical symplectic structure
 $$M=T^*S^n =\raisebox{0.3mm}{$Sl(n+1,\R)$}/\raisebox{-0.3mm}{$Gl_+(n,\R)$}.$$ 
Any symplectic maximal totally geodesic submanifold of dimension $2q$ is diffeomorphic to $$T^*S^q.$$
 All such submanifolds are in the same orbit of $Sl(n+1,\R)$;  this orbit is given by 
$$N_q=\raisebox{0.3mm}{$Sl(n+1,\R)$} / \raisebox{-0.3mm}{$S(Gl(q+1,\R)\times Gl(n-q,\R))$},$$ i.e. the space of pairs of supplementary spaces  in $\R^{n+1}$, with one of the spaces of dimension $q+1$. It is a symmetric symplectic space.

\item If  $A^2=-k^2\id$, we view as before $A=kJ$ with $J$ a complex structure and  we identify $\R^{2n+2}$ to $\C^{n+1}$ which is endowed
with the Hermitian  structure $\langle u,v \rangle= \Omega(u,Jv)-i\Omega(u,v)$ with complex  signature $(p+1,n-p)$. 
\begin{enumerate}
\item When $p=n$,  the group $G_A=SU(n+1)$ and the model  space form is the complex projective space 
     $$M=P_n(\C)=\raisebox{0.3mm}{$SU(n+1)$}/\raisebox{-0.3mm}{$U(n)$}.$$  
Every symplectic maximal totally geodesic submanifold of real dimension  $2q$ is diffeomorphic to $$P_q(\C).$$ There is only one $SU(n+1)$-orbit of symplectic maximal totally geodesic submanifolds for a given  dimension $2q$; it is given by 
      $$N_q=\raisebox{0.3mm}{$SU(n+1)$}/\raisebox{-0.3mm}{$S(U(q+1)\times U(n-q))$}.$$ 
The Radon transform in the case where  $q=n-1$ corresponds to the transform defined by antipodal submanifolds in $P_n(\C)$ (see \cite{sgbib:Helgason}).
\item When $1<p<n$, the group $G_A$ is $SU(p+1,n-p)$ and the model space form is 
        $$M=\raisebox{0.3mm}{$SU(p+1,n-p)$}/\raisebox{-0.3mm}{$U(p,n-p)$}.$$
Any symplectic maximal   totally geodesic submanifold of real dimension  $2q$ is of the form 
      $$ \raisebox{0.3mm}{$SU(p'+1,q-p')$}/\raisebox{-0.3mm}{$U(p',q-p')$}$$  
 for $p'<min(p,q)$. There is only one $SU(n+1)$-orbit of  symplectic maximal totally geodesic symplectic submanifolds for a given  dimension $2q$ and a given $p'$, $N_{q,p'}$ which is the symmetric symplectic space  given by
 $$
 N_{q,p'}=\raisebox{0.3mm}{$SU(p+1,n-p)$}/\raisebox{-0.3mm}{$S\left(U(p'+1,q-p')\times U(p-p',n-p-(q-p'))\right)$}.
 $$
\item When $p=0$, the group $G_A$ is $SU(1,n)$ and the model space form is  the complex hyperbolic space 
           $$M=\raisebox{0.3mm}{$SU(1,n)$}/\raisebox{-0.3mm}{$U(n)$}=H_n(\C).$$   
 Every symplectic totally geodesic submanifold of dimension  $2q$ is diffeomorphic  to $$H_q(\C).$$  
 There is one  $SU(1,n)$-orbit of  symplectic maximal totally geodesic symplectic submanifolds of a given dimension $2q$; it is given by 
            $$N=\raisebox{0.3mm}{$SU(1,n)$}/\raisebox{-0.3mm}{$S(U(1,q)\times U(n-q))$}.$$
  The Radon transform in the case where  $q=n-1$ corresponds to the Radon transform defined in the  Riemannian framework  with totally geodesic complex hypersurfaces in $H_n(\C)$  (see again \cite{sgbib:Helgason}).
  \end{enumerate}
 \item When $A^2=0$, we use as before the decomposition 
$$\R^{2n+2} = (V \oplus V') \oplus^{\perp_\Omega} W',$$
where $V = \Image (A)$, $V'$ is a $\Omega$-isotropic subspace supplementary to $\Ker (A)$ et $W'$ is the symplectic subpsace $\Omega$-orthogonal to $V \oplus V'$. The group $\tilde{G_A}$ is the following semi-direct product
$$S \cdot R = (SO_0(p,r-p)\times Sp(2m, \R)) \cdot (Mat(r\times 2m,\R)\cdot Sym(r,\R)),$$
and the model space is the symplectic product of a cotangent bundle, endowed with its canonical symplectic structure, and a symplectic vector space :
$$ M^{red} = T^*Q^{p,r-p}\times W = \raisebox{0.5mm}{$(S\cdot R)$}/ \raisebox{-0.5mm}{$(S' \cdot R')$},$$
where 
\begin{align*}
&S' = (SO_0(p-1,r-p)\times Sp(2m, \R)),\\
&R' = (Mat((r-1)\times 2m,\R)\cdot (\R\times Sym((r-1),\R))).
\end{align*}
Any maximal symplectic totally geodesic submanifold of dimension $2q$ is diffeomorphic to
$$T^*Q^{p', r'-p'} \times W'.$$
There is only one $(S \cdot R)$-orbit of symplectic maximal totally geodesic submanifolds for a given dimension $2q$, a given rank $r'$ and a given signature $p'$, denoted by $N_{q, r', p'}$, which is the following symmetric space
$$N_{q,r',p'} = \raisebox{0.5mm}{$(S\cdot R)$}/ \raisebox{-0.5mm}{$(S'' \cdot R'')$},$$
where 
\begin{align*}
&S'' = S(O(p',r'-p') \times O (p-p', (r-r')-(p-p')))_0\\
&\quad\quad \times (Sp(2m',\R) \times Sp (2(m - m'), \R)),\\
&R'' = (Mat(r'\times 2m',\R) \times Mat ((r-r') \times 2(m-m')))\\
&\quad \quad \cdot (Sym (r',\R) \times Sym(r-r',\R)),
\end{align*}
with $2m' := 2q-2r'+2$.

\end{itemize}
\end{proposition}
\section{Characterization in terms of totally geodesic submanifolds}\label{characterization}
We have seen in the previous section that our models $\Sigma_A$ of space forms admit as many symplectic maximal totally geodesic submanifolds as possible.
Reciprocally, we have :
\begin{theorem}
Let $(M,\omega)$ be a symplectic manifold of dimension $2n\ge 8$. Assume $(M,\Omega)$ to be endowed with a symplectic connection $\nabla$ and with a smooth parallel field $A$ of endomorphisms of the tangent bundle such that, $A^2=\lambda\Id$ with $\lambda$ constant and, for all $y\in M$, $A_y\in {\mathfrak{sp}}(T_yM,\omega_y)$. We assume here that the rank of $A$ is at least $3$.\\ 
Assume that, for any point $y\in M$ and any symplectic $A_y$-stable subspace $V$ of $T_yM$, there exists a totally geodesic submanifold $S$ of $M$ such that $y\in S$ and $T_yS=V$. Then $(M,\omega,\nabla)$ is locally symmetric and its canonical connection is of Ricci-type. 
\end {theorem}

\begin{proof}
Since the set of triples of tangent vectors to a point $y \in M$, $X, Y, Z \in T_y M$, such that $X, Y, Z, A_y X, A_y Y, A_y Z$ are linearly independent and span a symplectic $A_y$--stable subspace of  $T_y M$, is a dense open subset of $(T_y M)^3$, the condition implies that the curvature is given by
{\scriptsize{$$
R^\nabla_y(X,Y)Z=\alpha_y(Y,Z)X+\beta_y(X,Z)Y+\gamma_y(X,Y)Z+\alpha'_y(Y,Z)A_yX+\beta'_y(X,Z)A_yY+\gamma'_y(X,Y)A_yZ
$$}}
with $\alpha,\beta,\gamma,\alpha',\beta',\gamma'$ $2$-forms on $M$.\\
Now $R^\nabla_y(X,Y)=-R^\nabla_y(Y,X)$ and $\cyclic_{XYZ}R^\nabla_y(X,Y)Z=0$ imply 
{{$$
\beta=-\alpha,  \gamma(X,Y)=-\alpha(X,Y)+\alpha(Y,X),  \beta'=-\alpha', \gamma'(X,Y)=-\alpha'(X,Y)+\alpha'(Y,X)
$$}}
and the fact that $R^\nabla_y(X,Y)$ commutes with $A_y$ since $\nabla A=0$ implies $$\alpha(Y,Z)=\alpha'(Y,AZ).$$
The connexion $\nabla$ being symplectic, we also have $\omega(R^\nabla(X,Y)Z,T)=\omega(R^\nabla(X,Y)T,Z)$, i.e.
\begin{eqnarray*}
&&\alpha'(Y,AZ)\omega(X,T)-\alpha'(X,AZ)\omega(Y,T)+\bigl(\alpha'(Y,AX)-\alpha'(X,AY)\bigr)\omega(Z,T)\\
&&+\alpha'(Y,Z)\omega(AX,T)-\alpha'(X,Z)\omega(AY,T)+\bigl(\alpha'(Y,X)-\alpha'(X,Y)\bigr)\omega(AZ,T)
\end{eqnarray*}
is symmetric in $Z,T$, so that 
\begin{eqnarray}
&&\alpha'(Y,AZ)\omega(X,T)-\alpha'(X,AZ)\omega(Y,T)+2\bigl(\alpha'(Y,AX)-\alpha'(X,AY)\bigr)\omega(Z,T)\nonumber
\\
&&+\alpha'(Y,Z)\omega(AX,T)-\alpha'(X,Z)\omega(AY,T)-\alpha'(Y,AT)\omega(X,Z)\nonumber\\
&&+\alpha'(X,AT)\omega(Y,Z)-\alpha'(Y,T)\omega(AX,Z)+\alpha'(X,T)\omega(AY,Z)=0.\label{rel}
\end{eqnarray}
Choosing $T=AX$ and $X$ $\omega$-orthogonal to $Y,AY,Z,AZ$ yields $$\alpha'(Y,AZ)\omega(X,AX)=\alpha'(X,AX)\omega(Y,AZ).$$
Thus $\alpha'(Y,AZ)=f\omega(Y,AZ)$ and relation (\ref{rel}) becomes
\begin{eqnarray*}
&&f\omega(Y,AZ)\omega(X,T)-f\omega(X,AZ)\omega(Y,T)\\
&&+\alpha'(Y,Z)\omega(AX,T)-\alpha'(X,Z)\omega(AY,T)-f\omega(Y,AT)\omega(X,Z)\\
&&+f\omega(X,AT)\omega(Y,Z)-\alpha'(Y,T)\omega(AX,Z)+\alpha'(X,T)\omega(AY,Z)=0.
\end{eqnarray*}
Choosing $T=X$ and $X$ $\omega$-orthogonal to $Y,AY,Z,AZ$ yields
$$\alpha'(Y,Z)\omega(AX,X)+f\omega(X,AX)\omega(Y,Z)+\alpha'(X,X)\omega(AY,Z)=0$$
hence $\alpha'(Y,Z)=f\omega(Y,Z)+g\omega(Y,AZ)$ and relation (\ref{rel}) becomes
$$2g\bigl(\omega(Y,AZ)\omega(AX,T)-\omega(X,AZ)\omega(AY,T)\bigr)=0,
$$
hence $g=0$, so that $\alpha'(Y,Z)=f\omega(Y,Z)$ and 
\begin{eqnarray*}
R^\nabla_y(X,Y)Z&=&f(y)\bigl(\omega_y(Y,AZ)X-\omega_y(X,AZ)Y\bigr.\\
&&\bigl.+\omega_y(Y,Z)A_yX-\omega_y(X,Z)A_yY+2\omega_y(X,Y)A_yZ\bigr)
\end{eqnarray*}
which states that the connection is of Ricci-type. Bianchi's second identity implies that $f$ is a constant and the result follows.
\end{proof}

\section*{Acknowledgments}

We thank the organizers of the XXIV IFWGP meeting in Zaragoza for asking us to publish an expanded version of SG's talk.
This research  was partially supported by the  ARC  ``Complex, symplectic and contact
geometry, quantization and interactions'' of the communaut\'e fran\c caise de Belgique and by the PAI ``Dygest'' du gouvernement f\'ed\'eral de Belgique.
TG thanks the FRS-FNRS for support through a ``mandat d'Aspirant''.

\end{document}